\documentclass[submission,copyright,creativecommons]{eptcs}
\providecommand{\event}{ACT 2020} 
\usepackage{breakurl}             
\usepackage{underscore}           

\usepackage[utf8]{inputenc}
\usepackage[english]{babel}
\usepackage[T1]{fontenc}
\usepackage{amsmath}
\usepackage{hyperref}

\usepackage{subcaption}
\usepackage{tikz}
\usepackage{lipsum}

\newcommand{\tz}[2][1]{\vcenter{\hbox{\includegraphics[scale=#1]{graphics/#2.pdf}}}}

\usepackage{./macros}
\usepackage{./theoremstyles}

\title{The linear-non-linear substitution $2$-monad}

\author{Martin Hyland
\institute{DPMMS\\ University of Cambridge\\Cambridge, United Kingdom}
\email{M.Hyland@dpmms.cam.ac.uk}
\and
Christine Tasson
\institute{IRIF\\Université de Paris, CNRS\\ F-75013, Paris, France}
\email{\quad christine.tasson@irif.fr}
}
\def\titlerunning{LNL substitution monad}
\def\authorrunning{M. Hyland \& C. Tasson}

\begin{document}

\maketitle

\begin{abstract}
 We introduce a general construction on $2$-monads. We develop  background on maps of $2$-monads, their left semi-algebras, and colimits in $2$-category. Then, we introduce the construction of a colimit induced by a map of $2$-monads, show that we obtain the structure of a $2$-monad and give a characterisation of its algebras. Finally, we apply the construction to the map of $2$-monads between free symmetric monoidal and the free cartesian $2$-monads and combine them into a linear-non-linear $2$-monad.
\end{abstract}


A fundamental component of modern categorical logic is the treatment
of contexts.  In the standard approach, a sequence of variable
declarations is modelled by a product (of the interpretation of the
sorts or types) in some category. This point of view appeared early in
the approach to algebra by way of Lawvere
Theories~\cite{Lawverephd}. It has proved effective more widely, for
example in versions of “the internal language of toposes”. More subtle
interpretations of contexts are well established - fibrations over a
category of contexts in the case of Type
Theory~\cite{Lof1984,jacobs1999categorical}, symmetric modoidal
categories for Linear Logic~\cite{Girard87}; and over the last 20
years there has been increased awareness of the coherence issues. (In a small way, these are 
already apparent in Lawvere’s original treatment of algebra~\cite{Lawverephd}. For Type
Theory, the extent of the difficulties has been made clear by
Streicher~\cite{Streicher89} and Maietti~\cite{Maietti05}. For a
recent general take on  coherence issues see~\cite{Lumsdaine15}.)


This paper is a component of a project to consider (fragments of) the
differential lambda calculus~\cite{er:tdlc} from the point of view
of categorical algebra. For us the essence of such theories is the
idea of linear-non-linear contexts and to treat these we think it best
to set aside the standard treatment of contexts.  For ordinary algebra
we want to think not in terms of categories with products but in terms
of cartesian multicategories.  (Cartesian multicategories are the
abstract clones~\cite{Taylor1993} of the universal algebraists themselves formulated
abstractly. The connection is laid out in the early sections of~\cite{Hyland14monoid}.)
There are many current approaches to theories of multicategories,
driven at least to some extent by the great variety of extensions and
generalisations of the Theory of Operads which have proved of value
over the last 20 years. For our purposes we have firmly in mind (and
shall use in further papers) an approach via Kleisli Bicategories~\cite{fiore2018relative} as sketched in~\cite{Hyland14}.  This will involve an extension of the approach to variable
binding and substitution in abstract syntax~\cite{power2005binding, fiore2008cartesian,fioresubstitution,hirschowitz2007modules,Hyland17}.

For readers not familiar with substitution and variable binding, we
recall that it is based on the use of suitable $2$-monads $\mon$ on
$\Cat$ which extend to pseudo-monads on
$\Prof$~\cite{fiore2018relative}. In the corresponding Kleisli
bicategory, we can consider monads $M:A\slashedrightarrow \mon A$ and these can be
identified as generalised multicategories. For example, in case $\mon$
is the $2$-monad for symmetric monoidal categories, these are exactly
what are called many coloured operads or symmetric
multicategories~\cite{Curien12}. Similarly, in case $\mon$ is the
$2$-monad for categories with products, the monads in the Kleisli
bicategory are essentially many sorted algebraic
theories~\cite{Hyland14}.  In this paper, we show how to construct a
$2$-monad $\Dmon$ which would give rise to a notion of linear-non-linear multicategory or (what is for us the same thing) 
a linear-non-linear algebraic theory. 


As some motivation, we explain briefly what a linear-non-linear
algebraic theory consists of. Linear-non-linear theories have terms
which we can write $t({\bf x}^{\Svar} , {\bf y}^{\Cvar})$ where the
${\bf x}^{\Svar}$ is a collection of linear variables and
${\bf y}^{\Cvar}$ a collection of non-linear variables. To highlight
the point that terms are in context we write
${\bf x}^{\Svar}; {\bf y}^{\Cvar} \vdash t$ where now we think of
sequences of variables. (For the abstract definition it is best not to
keep the linear and non-linear variables apart with a stoup but doing
so makes the explanation clearer.)  Naturally equalities in the theory
will be given in context. The crucial issue is how the variables are
treated. Linear variables are to be handled as with symmetric operads:
one can permute variables in the context (Exchange) but one cannot
duplicate a linear variable (Contraction) nor create dummy such
(Weakening). On the other hand, all the standard manipulations are
available in the usual way for the non-linear variables. Substitution
for a linear variable is straightforward as shown on the lhs of the
Figure below.  All the interest arises from substitution for a
non-linear variable, see below on the rhs, the rule which makes all
variables in the substituted term non-linear. Of course one needs
variables declarations and these come both in linear form
$x^\Svar \vdash x$ and non-linear form $x^M \vdash x$. So by
substitution one can always regard a linear variable as non-linear but
not of course vice-versa.
\begin{figure}[h]
  \begin{subfigure}{0.4\textwidth}\centering
\begin{prooftree}
  \hypo{{\bf x}^\Svar, w^\Svar\ ;\  {\bf y}^\Cvar \vdash t}
  \hypo{{\bf u}^\Svar\ ;\ {\bf v}^\Cvar \vdash s}
  \infer2{{\bf x}^\Svar,{\bf  u}^\Svar\ ;\ {\bf v}^\Cvar, {\bf y}^\Cvar  \vdash t[s\slash w]}
\end{prooftree}
\label{fig1:sub}
\end{subfigure}
  \begin{subfigure}{0.4\textwidth}\centering
\begin{prooftree}
  \hypo{{\bf x}^\Svar\ ;\  {\bf y}^\Cvar, z^\Cvar \vdash t}
  \hypo{{\bf u}^\Svar\ ;\ {\bf v}^\Cvar \vdash s}
  \infer2{{\bf x}^\Svar\ ;\ {\bf y}^\Cvar,{\bf  u}^\Cvar,{\bf v}^\Cvar  \vdash t[s\slash z]}
\end{prooftree}
\label{fig2:sub}
\end{subfigure}
\end{figure}

It will be clear to those who like diagrams how to give a diagrammatic
notation for this. Non-linearity is like an infection: plugging a term
$s$ into a non-linear input causes all the inputs of $s$ to become
non-linear. The reader may find it helpful to compare this account
with the description of the corresponding 2-monad given in Section~\ref{subsec:dmon}.

In this paper we describe how to
obtain a 2-monad Q giving rise to linear-non-linear multicategories as above by
means of a particular general construction on $2$-monads in the sense of Cat-enriched monad theory~\cite{Dubuc1970}. Prima facie, the construction is not a universal one in a standard $2$-category of $2$-monads. All the same we are able precisely to characterise the $2$-category of algebras for the $2$-monad which we construct. This is a first step and further work will involve 2-dimensional monad theory in the sense of~\cite{blackwell19892mon}. Specifically, in the future, we shall address the question of extending our constructed $2$-monad on the $2$-category $\Cat$ of small categories to the corresponding bicategory $\Prof$ of profunctors or distributeurs~\cite{benabouprof,borceux1994handbook,benabou1967bicat}. We shall then use a resulting Kleisli bicategory~\cite{fiore2018relative} as the setting for an analysis of the foundations of the differential calculus as it appears in the differential $\lambda$-calculus~\cite{er:tdlc,blute2006differential,fiore2007differential}.  We shall also explain - probably in a paper separate from
the main development - how it is that the current standard semantics
for differential $\lambda$-calculus based on Linear Logic does indeed
give rise to a linear-non-linear semantics in our sense. (The naive
approach is not correct.)

Our project is based on $2$-monads on a $2$-category \(\cat\cK\) in the setting of the pioneering paper~\cite{blackwell19892mon}. Here, for a $2$-monad $\mon$ on $\cat\cK$, we follow the practice of that paper in writing $\Algs\mon$ for the $2$-category of strict $\mon$-algebras, strict $\mon$-algebra maps and $\mon$-algebra $2$-cells. We shall use more detailed information from~\cite{blackwell19892mon} in further papers.

In (enriched) categories of algebras for a monad, limits are easy and
it is colimits which are generally of more interest.  We assume
throughout that our ambient $2$-category $\cat\cK $ is cocomplete,
that our $2$-monads $\mon$ are such that the $2$-categories
$\Algs\mon$ are also cocomplete. In fact, we shall only need rather
innocent looking colimits in $\Algs\mon$, specifically the co-lax
colimit of an arrow. However, even that requires an infinite
construction~\cite{kelly1980}.  So it does not seem worth worrying
about minimal conditions for our results: we assume that we are in a
situation where all our $2$-categories are cocomplete. That happens
for example if our basic $2$-category is locally finitely presentable
and our monads are
finitary~\cite{KELLY1993163}.


\subsection*{Content}
In Section~\ref{sec:bg}, we first describe the background  on maps of $2$-monads (Subsection~\ref{subsec:mm}), left-semi algebras (Subsection~\ref{subsec:lsa}) and colimits (Subsection~\ref{subsec:col}), needed in our main Section~\ref{sec:cons}. There we define the colimits obtained from a map of monads (Subsection~\ref{subsec:def}) and exhibit their properties (Subsection~\ref{subsec:Qlsa}). Inspired by these properties, we define what we simply call the Structure $2$-category (Subsection~\ref{subsec:struc}).  We finally use (Subsection~\ref{subsec:Qmon}) the properties of the Structure $2$-category  to prove, in Theorem~\ref{th:Qmonad} that the colimit is a monad; and finally we prove our main Theorem~\ref{th:characterisation} which states that the Structure $2$-category is isomorphic to the $2$-category of strict algebras over the colimit monad. We end by spelling out the construction for two examples, the first one generates the left-semi algebra $2$-category (Proposition~\ref{prop:id}) and the second, what we call the linear-non-linear monad (Section~\ref{sec:lnl}) which was the original intention for developing this theory.

\section{Background}
\label{sec:bg}

\subsection{Maps of $2$-monads}
\label{subsec:mm}

The construction which we introduce here takes for its input a map $\Mmap:\Smon\rightarrow \Cmon$ of $2$-monads on $\cat\cK$. For clarity we stress that the usual diagrams commute on the nose. We rehearse some folklore related to this situation.

First, it is elementary categorical algebra that the monad map $\Mmap:\Smon\rightarrow \Cmon$ induces a $2$-functor
\(\lift \Mmap:\Algs\Cmon\rightarrow \Algs\Smon\).
On objects $\lift\Mmap$ takes an $\Cmon$-algebra $\Cmon X\rightarrow X$ to an $\Smon$-algebra $\Smon X\xrightarrow\Mmap \Cmon X \to X$.
It is equally evident that $\Mmap:\Smon\rightarrow \Cmon$ induces a $2$-functor
\(\shriek \Mmap:\kleisli\Smon\rightarrow \kleisli\Cmon\)
between the corresponding Kleisli $2$-categories. These 2-functors are essentially folklore. Given the evident relation between algebras for a (perhaps enriched!) monad
and modules for a ring,  $\Mmap^*$ can be called {\em restriction of scalars} and $\Mmap_!$ (or its extension see below)
{\em extension of scalars}. These connections are the driving force behind Durov’s PhD Thesis~\cite{durov2007}
which gives details of the phenomena.

We have the standard locally full and faithful comparisons:
\(  \kleisli\Smon \rightarrow \Algs\Smon\) and  \(\kleisli\Cmon \rightarrow \Algs\Cmon\).
Suppose we interpret $\shriek\Mmap$ as acting on the free algebras so that $\shriek\Mmap$ takes the free $\Smon$-algebra $\Smon^2A\xrightarrow\Smult \Smon A$ to the free $\Cmon$-algebra $\Cmon^2A\xrightarrow\Cmult \Cmon A$. 
Then we can see $\shriek\Mmap$ as a restricted left adjoint to $\lift\Mmap$ in the following sense. Given the free $\Smon$-algebra $\Smon^2A\xrightarrow\Smult \Smon A$ on $A$ and $\Cmon B\xrightarrow b B$ an arbitrary $\Cmon$-algebra, we have
\(\Algs\Smon(\Smon A,\lift \Mmap B) \simeq \Algs\Cmon(\shriek\Mmap\Smon A, B). \)
For $\shriek\Mmap(\Smon^2A\xrightarrow\Smult\Smon A)=\Cmon^2A\xrightarrow\Cmult\Cmon A$ and so both sides are isomorphic to $\cat\cK(A,B)$.

Any $\Smon$-algebra $\Smon A\xrightarrow a A$ lies in a coequalizer diagram in $\Algs\Smon$:
\(
\begin{tikzcd}
\Smon^2 A\ar[r,shift left=.75ex,"\Smult"]
  \ar[r,shift right=.75ex,swap,"\Smon a"]
  & \Smon A \ar[r,"a"]
  & a.
\end{tikzcd}
\)
So to extend $\shriek\Mmap$ to a full left adjoint $\shriek\Mmap:\Algs\Smon\Rightarrow \Algs\Cmon$ one has only to take the coequalizer of the corresponding pair in $\Algs\Cmon$:
\(\begin{tikzcd}
  \Cmon\Smon A\ar[r,shift left=.75ex,"\Cmult\Cmon \Mmap"]
  \ar[r,shift right=.75ex,swap,"\Cmon a"]
  & \Cmon A. 
\end{tikzcd}
\)
As it happens, we do not need the full left adjoint, but we shall need the unit of the adjunction given by the $\Smon$-algebra map $\Mmap_A$ from $\Smon^2 A \xrightarrow\Smult\Smon A$ to
\(\lift\Mmap\shriek\Mmap(\Smon^2 A\xrightarrow\Smult\Smon A)=\Smon\Cmon A\xrightarrow{\Mmap\Cmon} \Cmon^2 A\xrightarrow\Cmult \Cmon A\).

If $\tz{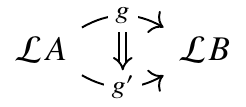}$ is an $\Smon$-algebra $2$-cell then the corresponding $2$-cell 
$\lift\Mmap\shriek\Mmap g\Rightarrow \lift\Mmap\shriek\Mmap g'$
is given by the composite $\tz{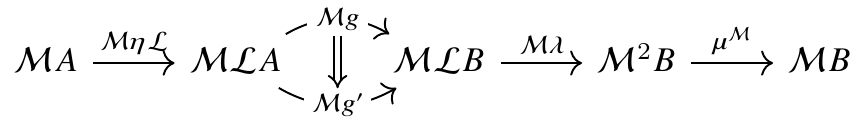}$ so that
\begin{equation}\label{eq:mm1}
\tz{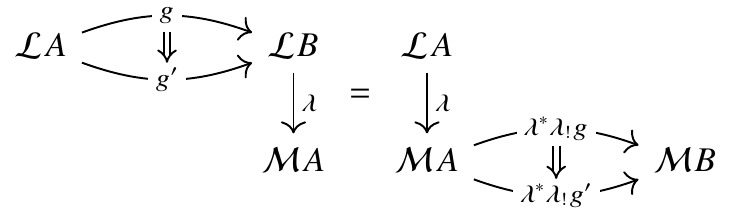}
\end{equation}

\subsection{Left-semi Algebras}
\label{subsec:lsa}

In this section we present some theory of a generalization of the notion of $\mon$-algebra for a $2$-monad $\mon$. In effect, it is a mere glimpse of an extensive theory of  semi-algebra structure, in the sense of structure "up to a retraction", a terminology well-established in computer science. 
We do not need to have this background in place for the results which we give in this paper: we give only what is required to make the paper comprehensible. However, some impression of what is involved can be obtained by looking at~\cite{garner2019vietoris} which gives some theory in the $1$-dimensional context.

\begin{definition}
  Let $\mon$ be a $2$-monad on a $2$-category $\cat\cK$. A left-semi
  $\mon$-algebra structure on an object $Z$ of $\cat\cK$ consists of a
  $1$-cell $\mon Z\xrightarrow z Z$ and a $2$-cell
  $\epsilon:z.\eta\Rightarrow \id_Z$ 
  satisfying the following $1$-cell and $2$-cell equalities:
  \begin{center}
  \begin{minipage}{0.25\linewidth}
     \begin{equation}\label{semialg1com}
       \tz{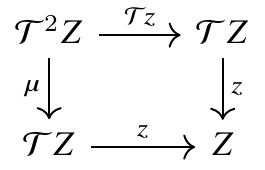}
    \end{equation}
  \end{minipage}\quad\vrule\quad
  \begin{minipage}{0.65\linewidth}
     \begin{equation}\label{semialg2com}
       \tz{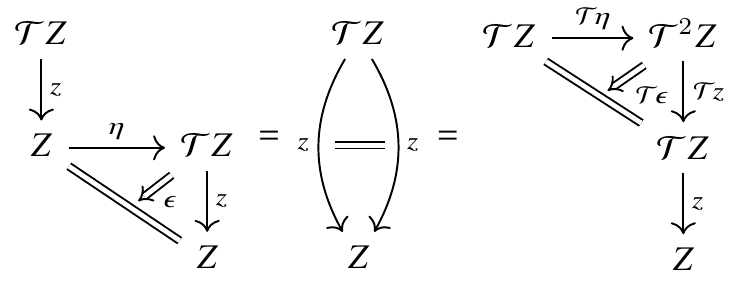}
    \end{equation}   
  \end{minipage}
  \end{center}
\end{definition}

\begin{remark}
  \begin{enumerate}
  \item The diagrams
    \begin{equation*}
      \tz{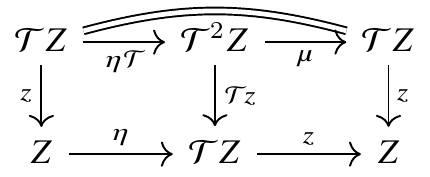}\qtand \tz{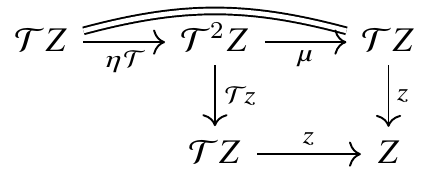}
    \end{equation*}
    demonstrate that Condition~\eqref{semialg1com} implies that the boundaries of the $2$-cells in Condition~\eqref{semialg2com} do match. 
  \item Condition~\eqref{semialg1com} is the standard composition for a strict $\mon$-algebra, while Condition~\eqref{semialg2com} is the unit condition for a colax $\mon$-algebra. 
  \end{enumerate}
\end{remark}

\begin{definition}
  Suppose that $\mon Z\xrightarrow z Z, \epsilon: z.\unit \Rightarrow \id_Z$ and $\mon W\xrightarrow w W, \epsilon: w.\unit \Rightarrow \id_W$ are left-semi $\mon$-algebras. A strict map from the first to the second consists of $p:Z\to W$ satisfying the following $1$-cell and $2$-cell equalities:
  \begin{center}
  \begin{minipage}{0.25\linewidth}
     \begin{equation}\label{semialg3com}
       \tz{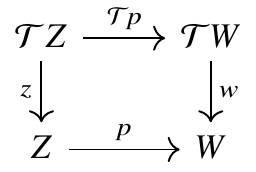}
    \end{equation}
  \end{minipage}\quad\vrule\quad
  \begin{minipage}{0.65\linewidth}
     \begin{equation}\label{semialg4com}
       \tz{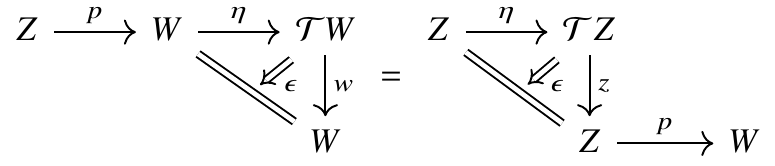}
    \end{equation}   
  \end{minipage}
  \end{center}
\end{definition}
\begin{remark}
  \begin{enumerate}
  \item The Condition~\eqref{semialg3com} with the naturality of $\unit$ imply that the boundaries of the $2$-cells in~\eqref{semialg4com} do match.
  \item The definition is the restriction to left-semi algebras of the evident notion of strict map of colax $\mon$-algebras.
  \item If $\mon Z\xrightarrow z Z, \epsilon: z.\unit\Rightarrow \id_Z$ is a left-semi algebra, then $\mon Z\xrightarrow z Z$ is a strict map to it from the free algebra $\mon^2 Z\xrightarrow \mult \mon Z$.
  \end{enumerate}
\end{remark}

\begin{proposition}
  Suppose that $\mon Z\xrightarrow z Z, \epsilon: z.\unit\Rightarrow \id_Z$ is a left-semi algebra. Then the composite $f:Z\xrightarrow\unit\mon Z\xrightarrow z Z$ is a strict endomap of the left-semi algebra.
\end{proposition}

Finally, we consider $2$-cells between maps of left-semi algebras.
\begin{definition}
  Suppose that $p,q:Z\to W$ are strict maps of left-semi algebras from $\mon Z\xrightarrow z Z, \epsilon: z.\unit\Rightarrow \id_Z$ to $\mon W\xrightarrow w W, \epsilon: w.\unit\Rightarrow \id_W$. A $2$-cell from $p$ to $q$ consists of a $2$-cell $\gamma:p\Rightarrow q$ such that the equality $\tz{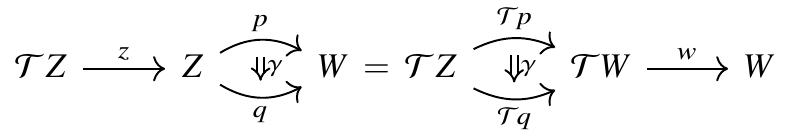}$ holds. 
\end{definition}
\begin{remark}
  Again, this is simply the restriction to the world of left-semi algebras of the definition of $2$-cells for colax $\mon$-algebras.
\end{remark}

\begin{proposition}
  Suppose that  $\mon Z\xrightarrow z Z, \epsilon: z\unit\Rightarrow \id_Z$ is a left-semi $\mon$-algebra, so that both $z.\unit$ and $\id_Z$ are strict endomaps. Then $\epsilon:z.\unit\Rightarrow\id_Z$ is a left-semi $\mon$-algebra $2$-cell.
\end{proposition}
At this point, it is straightforward to check that left-semi $\mon$-algebras, strict maps and $2$-cells form a $2$-category that we denote as $\lsalg\mon$.

Looking more closely at what we showed above we see that if we set $f=z\unit$, then we have $f=f^2$ and $\epsilon.f=\idm f=f.\epsilon$. So, in fact, we have the following.
\begin{proposition}
  Suppose that  $\mon Z\xrightarrow z Z, \epsilon: z.\unit\Rightarrow \id_Z$ is a left-semi $\mon$-algebra. Then, in the $2$-category $\lsalg\mon$, the $1$-cell $f$ and the $2$-cell $\epsilon:f\Rightarrow \id_Z$ equip the left-semi $\mon$-algebra with the structure of a strictly idempotent comonad.
\end{proposition}
Applying the evident forgetful $2$-functor we get that $f=f^2$ and $\epsilon: f\Rightarrow \id_Z$ equip $Z$ with the structure of a strictly idempotent comonad in the underlying $2$-category $\cat\cK$.

\begin{proposition}\label{prop:comonlsa}
  Suppose that $\mon X\xrightarrow x X$ is a $\mon$-algebra and $f=f^2:X\to X$ and $\epsilon: f\Rightarrow \id_X$ equip $X$ with the structure of a strictly idempotent comonad in $\Algs\mon$. Then $\mon X\xrightarrow x X\xrightarrow f X,\epsilon: fx\unit\Rightarrow \id_X$ is a left-semi $\mon$-algebra.
\end{proposition}
\begin{proof}[Proof sketch]
  The $1$-cell part is routine and the $2$-cell uses that $\epsilon$ is a $2$-cell in $\Algs \mon$.
\end{proof}

\begin{definition}
  Suppose that $\monS$ and $\mon$ are $2$-monads. A left-semi monad map  from the first to the second consists of $\lambda:\monS\to\mon$ satisfying the following equalities
  \begin{center}
  \begin{minipage}{0.4\linewidth}
     \begin{equation}\label{semialg5com}
       \tz{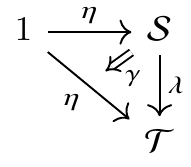}
    \end{equation}
  \end{minipage}\quad\vrule\quad
  \begin{minipage}{0.5\linewidth}
     \begin{equation}\label{semialg6com}
       \tz{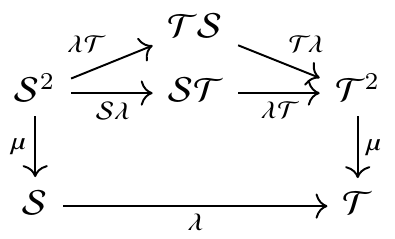}
    \end{equation}   
  \end{minipage}
  \begin{equation}\label{semialg7com}
    \tz{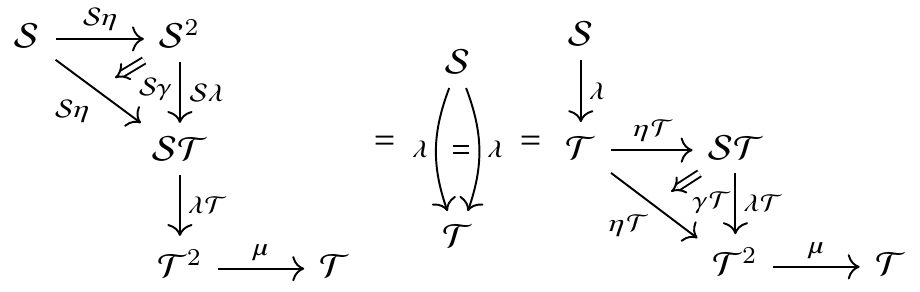}
  \end{equation}   
\end{center}
\end{definition}

\begin{proposition}\label{prop:mapmonlsa}
  Suppose that  $\mon Z\xrightarrow z Z, \epsilon: z.\unit\Rightarrow \id_Z$ is a left-semi $\mon$-algebra and $\monS\xrightarrow \lambda\mon,\gamma:\lambda.\unit\Rightarrow\unit$ is a left-semi monad map. Then $\monS Z \xrightarrow{\lambda_Z}\mon Z \xrightarrow z Z,\epsilon.\gamma:z.\lambda.\unit\Rightarrow\id_Z$ is a left-semi $\monS$-algebra.
\end{proposition}
\begin{proof}[Proof sketch.]
  The $1$-cell part is routine and the $2$-cell parts use the naturality of $\lambda$ to separate the two $2$-cells $\gamma$ and $\epsilon$.
\end{proof}

\subsection{Colax colimits induced by a map in $2$-category}
\label{subsec:col}

In this section we review the notion of colax colimits in a cocomplete $2$-category specialised to our context~\cite{BIRD19891,Lack2010}.

In the $2$-category $\cat\cK$, suppose that $\Dcell$ is a colax cocone $(\Scol,\Ccol,\Dcell)$ under the arrow $\Mmap$ (see Figure below
, left).
Then, for every $D$, composition with $\Dcell$ induces an isomorphism of categories between $\cat\cK(C,D)$ and the category of colax cocones under the arrow $\Mmap$ with objects $(f,g,\phi)$ (see Figure below, 
center) and $1$-cells $(f,g,\phi)\to (f',g',\phi')$ given by $2$-cells $f\xRightarrow \rho f'$ and $g\xRightarrow \sigma g'$ such that $\compC\rho\phi=\compC{\phi'}{\sigma.\lambda}$ (see Figure below 
right).
\begin{figure}[h]
\begin{subfigure}[t]{0.18\textwidth}
  \centering
  \[\tz{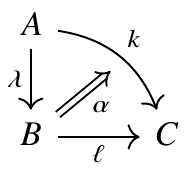}\]
\end{subfigure}\hfill
\begin{subfigure}[t]{0.18\textwidth}
  \centering
  \[\tz{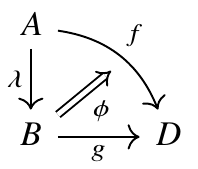}\]
\end{subfigure}
\begin{subfigure}[t]{0.60\textwidth}
  \centering
  \[\tz{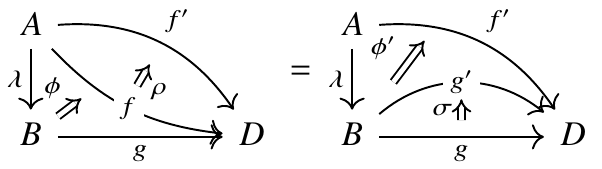}\]
\end{subfigure}
\end{figure}

This isomorphism of categories has two universal aspects, the first is $1$-dimensional and the second is $2$-dimensional:
\begin{itemize}
\item 
  for any $\tz{col2}$ there is a unique $r$ 
  such that $\tz{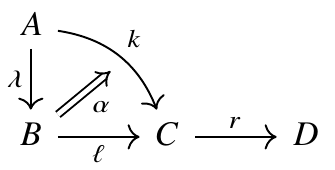}=\phi$
\item 
  for any $\tz{col3}$ there is a unique  $r\xRightarrow \tau r'$ such that
\end{itemize}
  \begin{equation}\label{eq:2cellcol}
  \tz{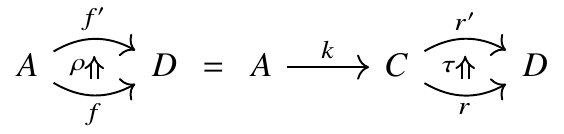}\qtand\tz{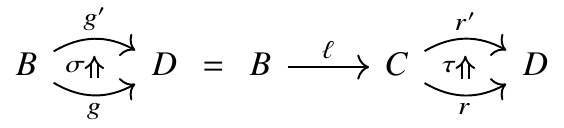}
\end{equation}

Although we will require colax colimits in the $2$-category of $\Algs\Smon$ where what happens is more subtle, we illustrate this definition by computing colax colimits in the $2$-category $\Cat$.
\begin{example}
  In $\Cat$, $A\xrightarrow\Mmap B$ is a functor between categories. The colax colimit under $\Mmap$ is a category $C$ which consists of separate copies of $A$ and $B$ together with, for every object $a\in A$,  new maps $\Mmap(a)\xrightarrow{\Dcell_a} a$, composition of such and evident identifications. Precisely, maps from $b\in B$ to $a\in A$ are given by $b\xrightarrow v\Mmap(a) \xrightarrow{\Dcell_A} a$ and $C(b,a)\simeq B(b,\Mmap(a))$.
\end{example}

\section{The colimit $2$-monad induced by a map of $2$-monads}
\label{sec:cons}

From now on, we assume that $\Smon$ is a finitary $2$-monad, so that $\Algs\Smon$ is cocomplete~\cite{KELLY1993163}.

\subsection{Definition of the colimit and its $2$-naturality}
\label{subsec:def}

\begin{definition}\label{prop:def}
  Suppose that $\Mmap:\Smon\to\Cmon$ is a map of $2$-monads. 
  Then the  colax colimit $(\Dmon X,\SDmap)$ under the induced $\Mmap_X:(\Smon X,\Smult)\to(\Cmon X,\Cmult)$ in $\Algs\Smon$  satisfies
  \begin{equation}\label{eq:cons1}\tz{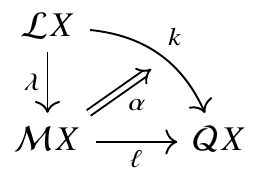}
  \end{equation}
\end{definition}

\begin{proposition}\label{prop1:def}
  The colax colimit $(\Dmon X,\SDmap)$ is natural in $(\Smon X,\Smult)$.
\end{proposition}
\begin{proof}[Proof sketch]
  Assume  $\tz{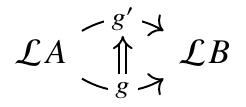}$ is an $\Smon$-algebra $2$-cell.
For each $1$-cell we get by $1$-cell naturality a cocone and so we get a unique maps $\widehat g$  mapping $\Dmon A$ to $\Dmon B$ arising from $1$-cell universality. We then have 
\[
\tz{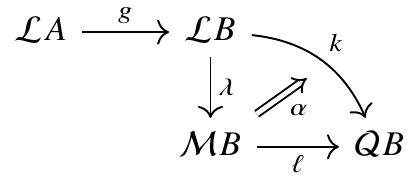}=\tz{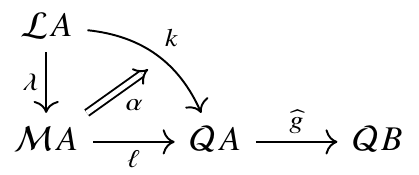}
\]
and similarly for $g'$ and $\widehat{g'}$. By $2$-cell universality~\eqref{eq:2cellcol}, we then get:
\[
\tz{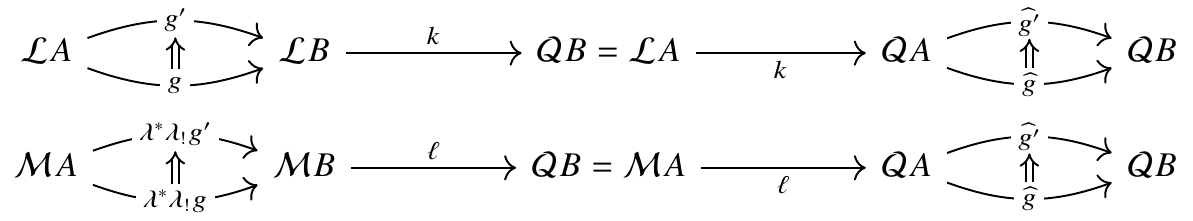}
\]
\end{proof}

\subsection{A left semi-algebra}
\label{subsec:Qlsa}
We explore the properties of $\Dmon X$ by considering $1$ and $2$ dimensional aspects of trivial cocones under $\Mmap$.
From the identity cocone under $\Mmap$, a unique $\Smon$-algebra map $\Ret$ arises by $1$-dimensional universality.
\begin{equation}
  \label{eq:cons3}
  \tz{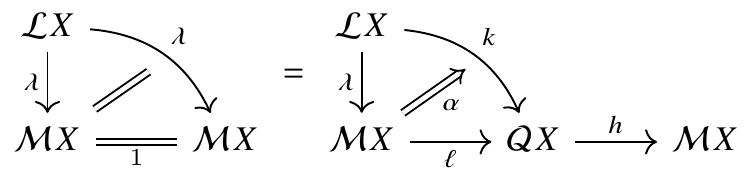} \qqtand \left\{
  \begin{array}{l}
    \Ret\,\Scol=\Mmap_X \\
    \Ret\,\Ccol=\id_{\Cmon X}\\
    \Ret . \Dcell =\idm\Mmap
  \end{array}\right.
\end{equation}
If  $\tz{cons00}$ is an $\Smon$-algebra $2$-cell,  then by $2$-dimensional universality $\Ret$ we will get:
\[
\tz{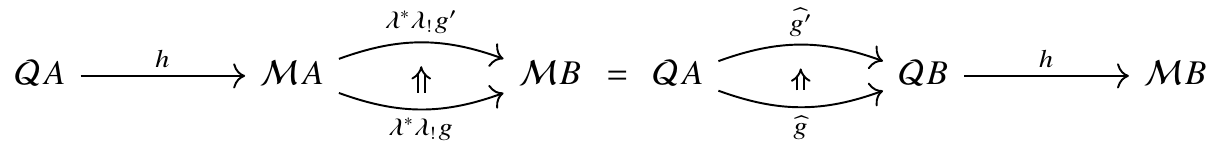}
\].
\\
From the $2$-cells $\idm\Ccol:\Ccol=\Ccol$ and $\Dcell:\Ccol\,\Mmap \Rightarrow\Scol$, arises a unique $\Algs\Smon$ $2$-cell $\Icell:\Ccol\,\Ret\Rightarrow \id_{\Dmon X}$ \sut
\begin{equation*}
  \tz{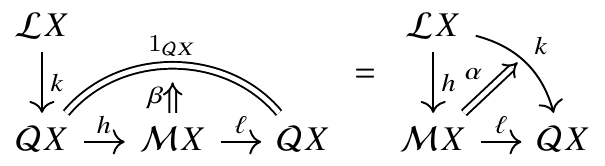}\qtand  \tz{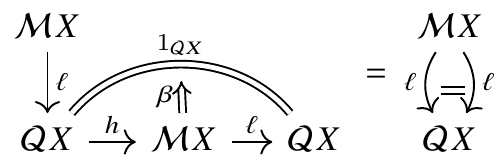}
\end{equation*}

Denote $\comon=\Ccol\,\Ret$. Then $\Dmon X$ is a $\Smon$-algebra and $\comon=\comon^2:\Dmon X\to\Dmon X$ and $\Icell:\comon\Rightarrow\id_{\Dmon X}$ equips $\Dmon X$ with the structure of a strictly idempotent comonad natural in $\Algs\Smon$ as
$\Icell .\Ccol =\idm\Ccol$, $\Icell .\Scol =\Dcell$, and thus  $\Ret .\Icell = \idm\Ccol$. We apply Proposition~\ref{prop:comonlsa} and get
\begin{proposition}\label{prop:Dstruct1}
  $\Smon\Dmon X\xrightarrow \SDmap \Dmon X \xrightarrow \Ret \Cmon X\xrightarrow \Ccol \Dmon X$ with $\Icell:\Ccol\Ret\SDmap\Sunit=\Ccol\Ret\Rightarrow \id_{\Dmon X}$ is a left-semi $\Smon$-algebra.
\end{proposition}

\begin{proposition}\label{prop:Dstruct2}
Assume $\CDmap$ denotes the map $\Cmon\Dmon X\xrightarrow{\Cmon \Ret}\Cmon^2 X\xrightarrow\Cmult\Cmon X\xrightarrow\Ccol\Dmon X$. Then $\Dmon X$ together with $\CDmap$ and $\CDmap\Cunit=\Ccol\Ret\xRightarrow\Icell \id_{\Dmon X}$ is a left-semi $\Cmon$-algebra.
\end{proposition}
\begin{proof}[Proof sketch]
The $2$-cell property relies on $\Icell.\Ccol=\idm\Ccol$ and $\Ret.\Icell=\idm\Ret$.
\end{proof}
As $\Mmap$ is a map of $2$-monads, it is a left-semi monad map. We apply Proposition~\ref{prop:mapmonlsa} and get
\begin{proposition}\label{prop:Dstruct3}
   $\Smon\Dmon X\xrightarrow{\Mmap\Dmon}\Cmon \Dmon X \xrightarrow{\Cmon\Ccol} \Cmon^2 X\xrightarrow \Cmult \Cmon X\xrightarrow\Ccol\Dmon X$ together with the $2$-cell $\Icell:\CDmap\,(\Mmap\Dmon)\Sunit=\Ccol\Ret\Rightarrow \id_{\Dmon X}$ is a left-semi $\Smon$-algebra.
\end{proposition}
The following is an immediate consequence of the definitions.
\begin{proposition}
  The left-semi $\Smon$-algebras of Proposition~\ref{prop:Dstruct1} and~\ref{prop:Dstruct3} are equal.
\end{proposition}

Let us recap the properties of $\Dmon X$. It is equipped with an $\Smon$-algebra structure $\SDmap$ and a left-semi $\Cmon$-algebra structure $\CDmap$ whose $2$-cell $\Icell$ lies in $\Algs\Smon$ and such that  the two resulting left-semi $\Smon$-algebra structures coincide.

In order to prove that $\Dmon$ is a $2$-monad (Theorem~\ref{th:Qmonad}) and that these properties characterise $\Dmon$-algebras (Theorem~\ref{th:characterisation}), we encapsulate the structure in a $2$-category. Given this structure on a general object $X$, we can build a map $\Dmon X\to X$ in a sufficiently functorial way that both theorems follow. What we need is the $1$-cell and $2$-cell aspects associated to these properties.

\subsection{The Structure category}
\label{subsec:struc}

Let us define the Structure category $\struc$
\begin{itemize}
\item an object  of $\struc$ consists of an object $X$ of $\cat\cK$ equipped with
  \begin{itemize}
  \item the structure  $\Smon X\xrightarrow w X$ of an $\Smon$-algebra
  \item the structure $\Cmon X\xrightarrow z X$, $\epsilon:z\,\Cunit= \comon\Rightarrow\id_X$ of a left-semi $\Cmon$-algebra
  \end{itemize}
  such that
  \begin{itemize}
  \item $\comon$ is an endomap of the $\Smon$-algebra $\Smon X\xrightarrow w X$ and $\epsilon$ is an $\Smon$-algebra $2$-cell
  \item the two induced left-semi $\Smon$-algebra structures, with structure maps $\Smon X\xrightarrow w X  \xrightarrow \comon X$ and $\Smon X\xrightarrow \Mmap \Cmon X\xrightarrow z X$, are equal
  \end{itemize}
\item a map in $\struc$ between objects $X$ and $X'$ equipped as above is a map $p:X\to X'$ in $\cat\cK$ which is both an $\Smon$-algebra map and a left-semi $\Cmon$-algebra map
\item a $2$-cell between two such maps $p$ and $q$ is a $2$-cell $p\Rightarrow p'$ which is both an $\Smon$-algebra and a left-semi $\Cmon$-algebra $2$-cell.
\end{itemize}

\begin{remark}
  \begin{enumerate}
  \item In the definition, the condition regarding the left-semi $\Smon$-algebra structures amounts to the claim that $\comon\,w=z\,\Mmap$. The equality of the $2$-cells is then automatic
  \item It is a consequence of the definition that $z:\Cmon X\to X$ is a map of $\Smon$-algebras. Indeed, if we consider the three following conditions, any two of them implies the third.
    \begin{itemize}
    \item $\comon$ is an endomap of $\Smon$-algebras, 
    \item $\comon\,w=\Mmap\, z$ 
    \item $z$ is a map of $\Smon$-algebras 
    \end{itemize}
  \end{enumerate}
\end{remark}

\begin{proposition}\label{prop:Dstruct}
  $\Dmon X$ together with  $\SDmap$, $\CDmap$ and $\Dcell$ is an object in $\struc$.
\end{proposition}

Assume  $X$ together with $w$, $z$, and $\epsilon$ is an object in $\struc$. Then we define $\Dmon X\xrightarrow x X$ to be the unique $\Algs\Smon$ map arising from the colax cocone
\begin{equation}\label{eq:cocone}
  \tz{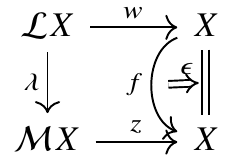}=\tz{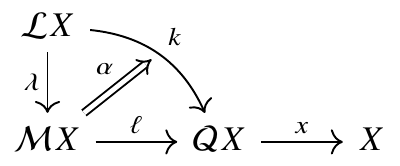}
\end{equation}
\begin{proposition}\label{prop:key}
  Assume  $X$ together with $w$, $z$, and $\epsilon$ is an object in $\struc$ and $x$ denotes the associated map. Then $x:\Dmon X\to X$ is a map in $\struc$ which is natural in $X$.
\end{proposition}
\begin{proof}[Sketch proof]
 Assume  $X'$ together with $w'$, $z'$, $\epsilon'$ in $\struc$ associated with  $x'$ and 
 $p\xRightarrow \rho q$ a $2$-cell in $\struc$. Then $\tz{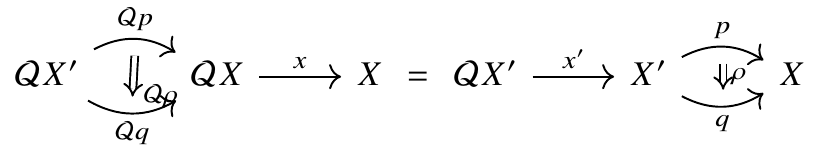}$ by $2$-cell universality.
\end{proof}

\subsection{The colimit is a monad}
\label{subsec:Qmon}

As $\Dmon X$ is an object in $\struc$ (Proposition~\ref{prop:Dstruct}), the induced map $\Dmon^2X\xrightarrow\Dmult \Dmon X$ is a map in $\struc$ (Proposition~\ref{prop:key}).

Assume  $(X,w,z,\epsilon)$ in $\struc$. Then the induced map $\Dmon X\xrightarrow x\Dmon X$ is a map in $\struc$. We apply the $1$-cell part of the naturality (Proposition~\ref{prop:key}) with $p=x$ and $x'=\Dmult$ and get
\[\tz{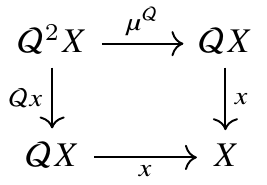}\qquad \text{in particular, setting $x=\Dmult$}\qquad\tz{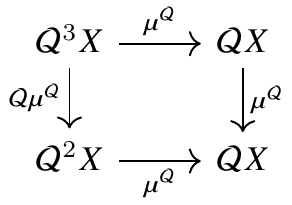}\]

\begin{theorem}
  \label{th:Qmonad}
  $\Dmon$ is a $2$-monad with multiplication $\Dmult$ and unit $X\xrightarrow \Sunit \Smon X\xrightarrow \Scol \Dmon X$.
\end{theorem}

\begin{proposition}\label{prop:monadmap}
  $\Smon \xrightarrow\Scol\Dmon$ is a map of monads.
\end{proposition}
\begin{proof}[Proof sketch]
  The unit aspect is by definition of $\Dunit$. As $\Scol$ is a map of $\Smon$-algebra and   $\Dmult\,\Scol=\SDmap$ by cocone equality~\eqref{eq:cocone}, we get the multiplication diagram.
\end{proof}

\begin{proposition}\label{prop:lsmonadmap}
  $\Cmon \xrightarrow\Ccol\Dmon$ is a left-semi map of monads.
\end{proposition}
\begin{proof}[Proof sketch]
Recall that  $\Ret\Ccol=\id$ and that   $\Dmult\,(\Ccol\Dmon)=\CDmap$ by cocone equality~\eqref{eq:cocone}.  Then,  the multiplication diagram~\eqref{semialg7com} follows since $\Dmult\,(\Ccol\Dmon)\,(\Smon\Ccol)=z\,(\Smon\Ccol)=\Ccol\,\Cmult\,(\Cmon\Ret)\,(\Cmon\Ccol)=\Ccol\,\Cmult$.

We define the unit $2$-cell $\gamma:\Ccol\,\Cunit\Rightarrow \Dunit$ in~\eqref{semialg5com} as
\[\tz{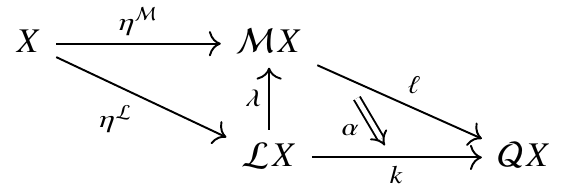}\]
We prove Equalities~\eqref{semialg6com}. Recall that $\Dcell=\Icell .\Scol$ and $\Icell.\Ccol=\idm\Ccol$.
 As $\Dmult\,(\Ccol\Dmon)=z=\Ccol\,\Cmult\,(\Cmon\Ret)$  and  $h.\Dcell=h .\Icell .\Scol=\idm\Ccol .\Scol$
\[\tz{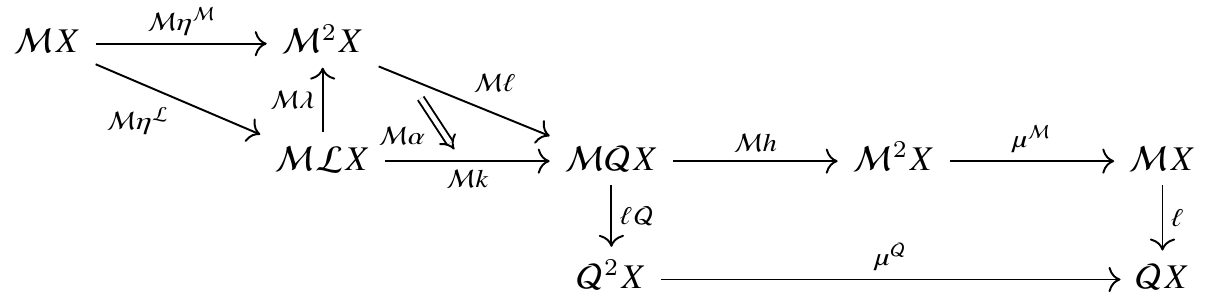}=\idm\Ccol.\]
As $\Dmult.\Dcell=\Icell.\SDmap$ (see Equality~\eqref{eq:cocone} with $x=\Dmult$), and as $\SDmap$ is an $\Smon$-algebra $u\,(\Sunit\Dmon)=\id_{\Dmon X}$ so the second $2$-cell equality follows: $\Dmult.\Dcell.(\Sunit\Dmon)\,\Ccol=\Icell.\SDmap\,(\Sunit\Dmon)\,\Ccol=\Icell.\Ccol=\idm\Ccol$.
\end{proof}

\begin{theorem}
  \label{th:characterisation}
  The $2$-category $\Algs\Dmon$ of algebras of the $2$-monad $\Dmon$  is isomorphic to the Structure category.
\end{theorem}
\begin{proof}[Proof sketch]
  It remains to prove the direct implication. Assume $\Dmon X\xrightarrow x X$ is a $\Dmon$-algebra.

\begin{itemize}
\item 
  Since $\Scol:\Smon\to\Dmon$ is a monad map, $w:\Smon X\xrightarrow \Scol \Dmon X \xrightarrow x X$ is an $\Smon$-algebra.
\item  By Propositions~\ref{prop:mapmonlsa}, since $\Ccol:\Cmon\to\Dmon$ is a left-semi monad map, $z:\Cmon X\xrightarrow \Ccol \Dmon X \xrightarrow x X$ is a left-semi $\Cmon$-algebra with $2$-cell $\Dcell$ where  we denote $\comon_x= z\,\Cunit$
  \begin{equation}\label{eq:carac}
    \tz{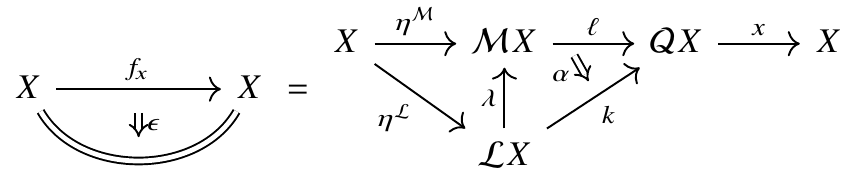}
  \end{equation}
\item  We know that $\Ret\,\Ccol=\Mmap$ and $\Ret\Ccol=\id_{\Dmon X}$ and $z=x\,\Ccol$ is a left-semi $\Cmon$-algebra. We deduce  $\Smon X\xrightarrow w X\xrightarrow \comon_x X= \Smon X\xrightarrow \Mmap \Cmon X\xrightarrow z X$ using the following.
  \[\tz{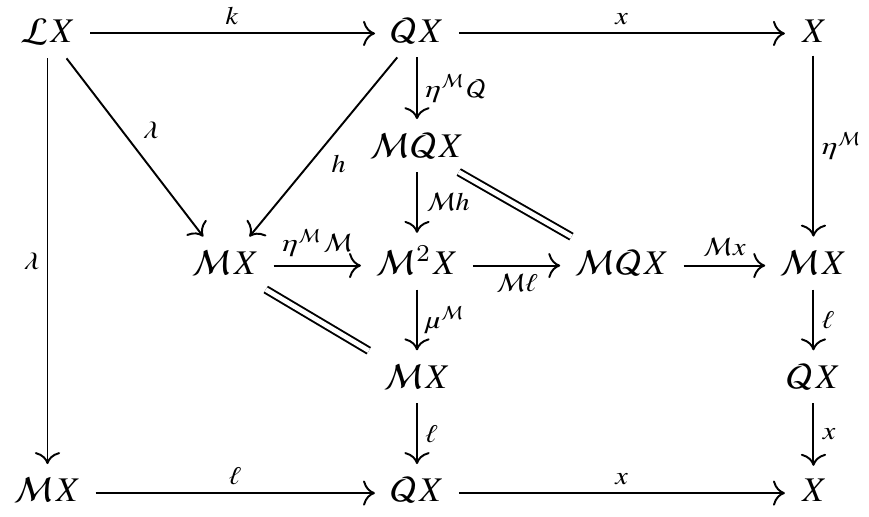}\]
\item   We prove that $\Rcell$ is in $\Algs \Smon$. 
  We first remark that $x.\Icell=\Rcell.x$. Indeed, by naturality of $\Cunit$ and of $\Dcell$, we have $\Dcell.\Sunit\, x=(\Dmon x).\Dcell.\Sunit$. Because $x$ is a $\Dmon$-algebra, $x.\Dcell.\Sunit\, x=x\,(\Dmon x).\Dcell.\Sunit=x\,\Dmult.\Dcell.\Sunit$ and we conclude as $\Dmult.\Dcell.\Sunit=\Icell$.

  Then, as $\Icell$ is an $\Smon$-algebra $2$-cell by construction and $x$ is a $\Smon$-algebra, so that $\Rcell.x$ is an $\Smon$-algebra $2$-cell. This can be represented by the lhs $2$-cell equality which results in the rhs equality by precomposition by $\Smon\Dunit$.   This proves that $\Rcell$ is an $\Smon$-algebra $2$-cell.
  \[\tz{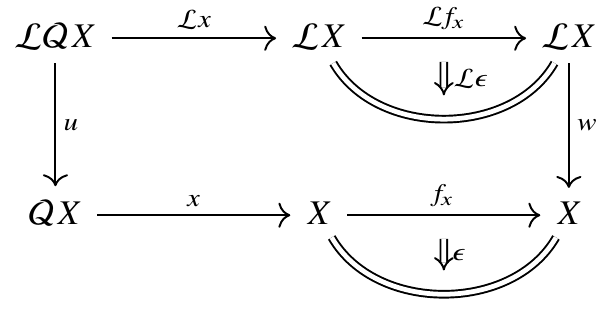} \qquad \qquad\tz{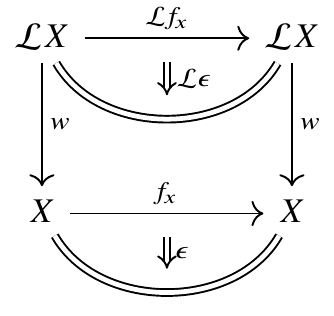}\]

\end{itemize}
\end{proof}

  Our analysis of the $2$-monad $\Dmon$ involved consideration of left-semi $\Cmon$-algebras. We can immediately say something about them. Suppose that $\Cmon^+$ is the result of applying our construction to the map 
$\unit:\monI\to\Cmon$ of monads given by the unit. By Theorem~\ref{th:characterisation}, we deduce the following.
\begin{proposition}\label{prop:id}
 $\Algs{\Cmon^+}$ is isomorphic to  $\lsalg\Cmon$
\end{proposition}
So the $2$-category of left-semi $\Cmon$-algebras is in fact monadic over the base $\cat\cK$.

\section{The linear-non-linear $2$-monad}
\label{sec:lnl}

In this section, we show how our theory applies in the case of most
immediate interest to us. We take for $\Smon$ the $2$-monad for
symmetric strict monoidal categories: we give a concrete presentation
in Subsection~\ref{subsec:smon}. We take for $\Cmon$ the $2$-monad for categories
with strict finite products: we give a concrete presentation
in Subsection~\ref{subsec:cmon}. There is an evident map of monads
$\Smon\to\Cmon$ and in Subsection~\ref{subsec:dmon}, we describe the $2$-monad
$\Dmon$ obtained by our construction.

In further work we shall develop general theory to show that this
$\Dmon$ in particular extends from $\CAT$ to profunctors. This gives a
notion of algebraic theory in the sense of Hyland~\cite{Hyland14} and
we shall use that to handle the linear and non-linear substitutions
appearing in differential lambda-calculus~\cite{er:tdlc}.

\subsection{The $2$-monad for symmetric strict monoidal categories}\label{subsec:smon}

For a category $A$, let $\Smon A$ be the following category. The objects
are finite sequences $\seq{a_i}{i\in \nset n}$ with $n\in\N$ and
$a_i\in A$. The morphisms
\[\seq{a_i}{i\in \nset n}\to \seq{a'_j}{j\in \nset m}\]
consist of a bijection $\sigma:\nset n\to\nset m$ (so $n$ and $m$ are equal) and for each $j\in\nset m$ a map $a_{\sigma(j)}\to a'_j$ in $A$. The identity and composition are evident.

$\Smon$ extends readily to a $2$-functor on $\CAT$ and it has the structure of a $2$-monad where $\Sunit:A\to\Smon A$ takes $a$ to the singleton $\seq a{}$ and $\Smult:\Smon^2 A\to\Smon A$ acts on objects by concatenation of sequences.

Each $\Smon A$ has the structure of a symmetric monaidal category: the unit is the empty sequence and tensor product is given by concatenation. One can check directly that $\Sunit:A\rightarrow \Smon A$ makes $\Smon A$ the free symmetric strict monoidal category on $A$. Moreover to equip $A$ with the structure of a symmetric strict monoidal category is to give $A$ an $\Smon$-algebra structure. Maps and $2$-cells are as expected so we identify $\Algs\Smon$ as the $2$-category of strict monoidal categories, strict monoidal functors and monoidal $2$-cells.

\subsection{The $2$-monad for categories with products}\label{subsec:cmon}
For a category $A$, let $\Cmon A$ be the following category. The objects
are finite sequences $\seq{a_i}{i\in \nset n}$ with $n\in\N$ and
$a_i\in A$. The morphisms
\[\seq{a_i}{i\in \nset n}\to \seq{a'_j}{j\in \nset m}\]
consist of a map $\phi:\nset m\to\nset n$  and for each $j\in\nset m$ a map $a_{\phi(j)}\to a'_j$ in $A$. The identity and composition are evident.

$\Cmon$ extends readily to a $2$-functor on $\CAT$ and it has the structure of a $2$-monad where $\Cunit:A\to\Cmon A$ takes $a$ to the singleton $\seq a{}$ and $\Cmult:\Cmon^2 A\to\Cmon A$ acts on objects by concatenation of sequences.

Each $\Cmon A$ has the structure of a category with strict products: the terminal object is the empty sequence and product is given by concatenation. Again, one can check directly that $\Cunit:A\rightarrow \Cmon A$ makes $\Cmon A$ the free category with strict products on $A$. Again, to equip $A$ with the structure of a category with strict products is to give $A$ a $\Cmon$-algebra structure. Maps and $2$-cells are as expected so we identify $\Algs\Cmon$ as the $2$-category of categories with strict products, functors preserving these strictly and appropriate  $2$-cells.

\subsection{The $2$-monad for  linear-non-linear substitution}\label{subsec:dmon}

There is a map $\Mmap:\Smon\to\Cmon$ which on objects takes $\seq{a_i}{i\in \nset n}\in\Smon A$ to  $\seq{a_i}{i\in \nset n}\in\Cmon A$ and includes the maps in $\Smon A$ into those in $\Cmon A$ in the obvious way. It accounts for the evident fact that every category with strict products is a symmetric strict monoidal category. We describe the $2$-monad $\Dmon$ obtained from  $\Mmap$ by our colimit construction.

For a category $A$, $\Dmon A$ is the following category. The objects are
$\seq{a_i^{\epsilon_i}}{i\in \nset n}$ with $n\in\nset n$, $a_i\in A$ and the indices $\epsilon_i$ chosen from the set $\{\Svar,\Cvar\}$ ($\Svar$ indicates linear and $\Cvar$ non-linear). For $a=\seq{a_i^{\epsilon_i}}{i\in \nset n}$, write $\Svar_a$ for $\{i\,|\, \epsilon_i=\Svar\}$.  Then a  morphism
\[\seq{a_i}{i\in \nset n}\to \seq{a'_j}{j\in \nset m}\]
is given by first a map $\phi:\nset m\to\nset n$ satisfying the condition
\[\phi^{-1}(\Svar_a) \subseteq \Svar_{a'}\qtand \phi_{|\phi^{-1}(\Svar_a)}:\phi^{-1}(\Svar_a)\to\Svar_a\text{ is a bijection;} \]
and secondly by for each $j\in\nset m$, a map $a_{\phi(j)}\to a'_j$ in $A$.

$\Dmon$ extends readily to a $2$-functor on $\CAT$ and it has the structure of a $2$-monad as follows. The unit $\Dunit:A\to\Dmon A$ takes $a\in A$ to $\seq{a^\Svar}{}$. The multiplication $\Dmult: A\to\Dmon A$ acts by concatenating the objects and with the following behaviour on indices: objects of $\Dmon^2 A$ have shape
\[\seqs{\seqs{\dots}\dots \seqs{\dots a^\epsilon\dots}^\eta\dots\seqs{\dots}}\]
so that each $a\in A$ has two indices; in the concatenated string in $\Dmon A$, $a$ has index $\Svar$ just when both $\epsilon$ and $\eta$ are $\Svar$.

One can now readily see the structure on $\Dmon A$ involved in its definition.
\begin{itemize}
\item $\Dmon A$ is clearly an $\Smon$-algebra and $\Scol:\Smon A\to\Dmon A$ sends $\seqs{a_1,\dots,a_n}$ to $\seqs{a_1^\Svar,\dots,a_n^\Svar}$
\item $\Ccol:\Cmon A\to\Dmon A$ sends $\seqs{a_1,\dots,a_n}$ to $\seqs{a_1^\Cvar,\dots, a_n^\Cvar}$ given by the identity on $\nset n$ and is evidently an $\Smon$-algebra map
\item $\Dcell:\Ccol\Mmap\to\Scol$ is given for each $\seqs{a_1,\dots,a_n}\in\Smon A$ by the map $\seqs{a_1^\Cvar,\dots,a_n^\Cvar}\to\seqs{a_1^\Svar,\dots,a_n^\Svar}$ given by the identity on $\nset n$ and identities $a_i\to a_i$ for each $i$.
\\
It is also easy to see $h:\Dmon A\to \Cmon A$: it sends $\seqs{a_1^{\epsilon_1},\dots,a_n^{\epsilon_n}}$ to $\seqs{a_1,\dots,a_n}$. It should now be straightforward for the reader to identify the $2$-cell $\beta$ and deduce that $\Dmult$ is just as described.
\end{itemize}

Now, we can use our Theorem~\ref{th:characterisation} to give a description of what a $\Dmon$-algebra is in this case. It is an object of our structure category described in Subsection~\ref{subsec:struc}. That means it is a symmetric monoidal category $X$ equipped with a strictly idempotent comonad $f:X\to X$ with $\epsilon:f\Rightarrow \id_X$ and with that structure in the $2$-category of symmetric monoidal categories and strict maps; it is such that the full subcategory of fixpoints of $f$ is equipped with the structure of a category with products; moreover the effect of tensoring objects of $X$ and then applying $f$ is equal to that of first applying $f$ and then taking the product.

\subsection{Next steps}\label{nextstep}

  Starting from the observation that the $2$-monad $\Smon$ for strict monoidal categories and the $2$-monad $\Cmon$ for categories with strict products can be combined into a $2$-monad $\Dmon$ mixing the two related structures, we have introduced a new notion for combining $2$-monads as the colimit of a map of monads. We have proved that our construction gives rise to a $2$-monad in Theorem~\ref{th:Qmonad} and characterised its algebras in Theorem~\ref{th:characterisation}.

  Our next step will be to give conditions under which $\Dmon$ admits an extension to a pseudomonad on $\Prof$~\cite{fiore2018relative}. 
We draw attention to the following issue which we need to address. It is clear from~\cite{fiore2018relative} that the $2$-monad $\Smon$  for symmetric strict monoidal categories and $\Cmon$ for categories with strict products admit extensions to pseudomonads on $\Prof$. However, we cannot use our colimit construction at this level as we only have access to bicolimits. All the same, the characterisation of Theorem~\ref{th:characterisation} can be reworked so as  to describe pseudo $\Dmon$-algebras. Then one can show that the presheaf construction has a lifting to pseudo $\Dmon$-algebras and so deduce by~\cite{fiore2018relative} the wanted extension of $\Dmon$ to $\Prof$.

 The extension of $\Dmon$ to $\Prof$ will give a notion of linear-non-linear multicategory which will serve as a basis for describing the substitution structure at play in differential $\lambda$-calculus~\cite{fiore05}. In parallel, we shall compare our approach to existing approaches to the combination of linearity and non-linearity which arises from Linear Logic~\cite{Benton94}. We hope to show that starting from the models of Benton~\cite{Benton94} or Blute-Cokcett-Seely~\cite{blute2006differential}, we can obtain a $\Dmon$-algebra (or at least a $\Dmon$-multicategory) which accounts for the usual practice of modelling linear-non-linear calculi. 

\bibliographystyle{eptcs}
\bibliography{references}


\end{document}